\documentclass[10pt,a4paper,twoside]{amsart}
\usepackage{amsmath,amsfonts,amsthm,amsopn,color,amssymb,enumitem}
\usepackage{palatino}
\usepackage{graphicx}
\usepackage[colorlinks=true]{hyperref}
\hypersetup{urlcolor=blue, citecolor=red, linkcolor=blue}

\newcommand{\e}{\varepsilon}
\newcommand{\C}{\mathcal{C}}
\newcommand{\F}{\mathcal{F}}

\newcommand{\R}{\mathbb{R}}

\newcommand{\dd}{\Gamma}
\newcommand{\RN}{{\mathbb{R}^N}}

\renewcommand{\d }{\delta }

\def\bbm[#1]{\mbox{\boldmath $#1$}}
\newcommand{\beq }{\begin{equation}}
\newcommand{\eeq }{\end{equation}}

\newtheorem{theorem}{Theorem}[section]
\newtheorem{lemma}[theorem]{Lemma}
\newtheorem{example}[theorem]{Example}
\newtheorem{definition}[theorem]{Definition}
\newtheorem{proposition}[theorem]{Proposition}
\newtheorem{remark}[theorem]{Remark}
\newtheorem{corollary}[theorem]{Corollary}

\title[A note on the Poincar{\'e} inequality]
{A note on the uniformity of the constant in the Poincar{\'e}
inequality}

\author[David Ruiz]{David Ruiz$^1$}
\address{$^1$Dpto. An{\'a}lisis Matem{\'a}tico, Facultad de Ciencias, Avda. Fuentenueva s/n, Granada, 18071 Spain.}

\thanks{David Ruiz has been supported by the Spanish Ministry of
Science and Innovation under the Grant MTM2011-26717, and by J.
Andaluc\'{\i}a (FQM 116).}

\email{daruiz@ugr.es}
\date{}



\begin{document}

%

\maketitle

\begin{center} Dedicated to A. Ambrosetti, a guide who definitively changed my life


\end{center}

\begin{abstract}
The classical Poincar{\'e} inequality establishes that for any bounded
regular domain $\Omega\subset \RN$ there exists a constant
$C=C(\Omega)>0$ such that
$$ \int_{\Omega} |u|^2\, dx \leq C \int_{\Omega} |\nabla u|^2\, dx
\ \ \forall u \in H^1(\Omega),\ \int_{\Omega} u(x) \, dx=0.$$ In
this note we show that $C$ can be taken independently of $\Omega$
when $\Omega$ is in a certain class of domains. Our result
generalizes previous results in this direction.
\end{abstract}

\section{Introduction}

This paper is concerned with the following classical result, known
as Poincar{\'e} inequality (or Poincar{\'e}-Friedrichs inequality):

\begin{theorem} \label{poincare} Let $\Omega \subset \RN$ be a bounded domain
satisfying the interior cone condition, and $p>1$. Then, there
exists a constant $C=C(\Omega)>0$ such that
$$ \int_{\Omega} |u|^p\, dx \leq C \int_{\Omega} |\nabla u|^p\, dx \
\forall u \in W^{1,p}(\Omega),\ \int_{\Omega} u(x) \, dx=0.$$

\end{theorem}

The proof is very simple and standard. Assume, reasoning by
contradiction, that $u_n \in W^{1,p}(\Omega)$ is a sequence such
that

$$\int_{\Omega} u_n\, dx=0,\ \int_{\Omega} |u_n|^p\, dx=1,\ \int_{\Omega} |\nabla u_n|^p\, dx
\to 0.$$

By the Rellich-Kondrachov theorem, one obtains that $u_n
\rightharpoonup u$, where:

$$\int_{\Omega} u\, dx=0,\ \int_{\Omega} |u|^p\, dx=1,\ \int_{\Omega} |\nabla u|^p\, dx
= 0.$$ But if $\nabla u =0$ almost everywhere in a domain, then
$u$ must be constant (see Chapter 9 of \cite{brezis}, for
instance). And this yields the desired contradiction.

\medskip As one can observe, two main ingredients come into play in the proof:
first, the compactness of the $L^p$ embedding. For that one needs
the interior cone condition on $\Omega$ and its boundedness.
Second, the connectedness of the domain.

In this paper we try to answer the following question: under which
conditions on $\Omega$ we can assure the existence of $C$
independently of the set $\Omega$? The above discussion suggests
that one should need three kind of assumptions:

\begin{enumerate}
\item a uniform bound on $\Omega$, \item an interior cone
condition (with a fixed cone), \item and, in a certain sense, a
uniform connectedness assumption on $\Omega$. \end{enumerate}

The goal of this paper is to prove that this is indeed the case.

The uniformity of the Poincar{\'e} constant is useful in some free
boundary problems or shape optimization problems, see \cite{6,7}.
Our motivation comes from the derivation of a uniform local
version of the Moser-Trudinger inequality, which is very useful in
the variational study of the mean field equation on compact
surfaces (see \cite{chen-li}, and also \cite{djlw}, \cite{dm}). We
hope to use the results of this paper to give a proof for such
inequality in a forthcoming work.

Before going on with our exposition, let us mention some previous
results. In the literature there are some results concerning the
dependence of the Poincar{\'e} constant on $\Omega$: see for instance
\cite{chenais, meyer}. However, the assumptions required there
involve uniform $C^1$ regularity on $\Omega$ and other conditions.
Here we are interested in a less regular framework. On the
contrary, in \cite{chenais, meyer} more information is given,
related to the solution of the associated Neumann eigenvalue
problem.

In \cite{chakib} the uniformity of the Poincar{\'e} constant is proved
for uniformly Lipschitz domains. Let us point out that Lipschitz
domains must lie (locally) on one side of the boundary. The unit
ball with a segment removed, for instance, is not Lipschitz.
Moreover, domains with interior cusps are not covered, for
instance. In both cases the Poincar{\'e} inequality holds, which makes
one think that the uniform Lipschitz assumption is too
restrictive.

Throughout the paper we shall use the following notation:

$$ \forall \, \e>0,\ \Omega^\e=\{x \in \Omega:\ d(x, \partial
\Omega) >\e \}.$$

We will consider the following hypotheses.

\begin{enumerate}[label=(h\arabic*), ref=(h\arabic*)]
\setcounter{enumi}{0}

\item \label{h1} There exists $R>0$ such that $\Omega \subset
B(0,R)$.

\item \label{h2} There exists a fixed finite cone $\C$ such that
each point $x \in
\partial \Omega$ is the vertex of a cone $\C_x$ congruent to $\C$ and contained in $\Omega$.

\item \label{h3} There exists $\delta_0 >0$ such that for any
$\delta \in (0, \delta_0)$, $\Omega^{\delta}$ is a connected set.
\end{enumerate}

This last condition is a sort of ``uniform connectedness
assumption". We will show in Example \ref{ejemplo} that a
condition in this direction is necessary in our results. Roughly
speaking, condition \ref{h3} avoids the existence of arbitrarily
narrow junctions between two regions in $\Omega$.

The main result of this note is the following.

\begin{theorem} \label{main} For any $p>1$ there exists a constant $C>0$ such that for any
domain $\Omega$ satisfying \ref{h1}-\ref{h3},

$$ \int_{\Omega} |u|^p\, dx \leq C \int_{\Omega} |\nabla u|^p\, dx \
\ \forall u \in W^{1,p}(\Omega),\ \int_{\Omega} u(x) \, dx=0.$$
Here $C$ depends only on $p$, $R, \d_0$ and the cone $\C$ given in
conditions \ref{h1}-\ref{h3}.

\end{theorem}

We can also give some generalizations of this result, that will be
discussed in Section 3. In particular, we can prove the uniform
estimate in the following situation:

\begin{corollary} \label{coro} For any $p>1$, $R>0$ and $r >0$ there exists a constant $C>0$ such that for any
connected set $K \subset B(0,R)$,

$$ \int_{K_r} |u|^p\, dx \leq C \int_{K_r} |\nabla u|^p\, dx \
\ \forall u \in W^{1,p}(K_r),\ \int_{K_r} u(x) \, dx=0.$$ Here
$K_r$ denotes the domain $K_r =\{ x \in \RN:\ d(x, K)<r\}$.
Observe that $C$ is independent of the choice of $K$.

\end{corollary}

The uniformity of the Poincar{\'e} constant on this type of sets was
the original motivation of our work. We think that this corollary
could be one of the ingredients of a future local version of the
Moser-Trudinger inequality.

\medskip Our results improve that of \cite{chakib} since  a
connected and uniformly Lipschitz domain satisfy our conditions
\ref{h1}-\ref{h3} (see the final Appendix). Moreover, we can cover
many situations in which the domains are not Lipschitz, as
explained above.

Our proof is also quite different from that of \cite{chakib},
which uses different notions of convergence of sets to pass to a
limit in an argument by contradiction. Here we give a direct
proof, based on two main tools. The first one is the following
property:

\begin{enumerate}[label=(Q), ref=(Q)]

\item \label{q} For any $\e >0$, there exists $\delta>0$
independent of $\Omega$ such that $|\Omega \setminus
\Omega^{\delta}| < \e$.

\end{enumerate}

Let us point out that any domain $\Omega$ satisfies that condition
with $\delta$ depending on $\Omega$; this is due to the continuity
from above of the Lebesgue measure. We shall see that conditions
\ref{h1} and \ref{h2} imply the uniformity of the choice of
$\delta$.

The second ingredient of our arguments is the observation on the
uniformity of the constant in the Sobolev embedding under our
hypotheses.

The rest of the paper is organized as follows. In Section 2 we
first give some preliminary results, some of them well-known,
which are useful later in the proof of Theorem \ref{main}. Section
3 is devoted to the discussion of some extensions of our results.
In particular, Corollary \ref{coro} will be proved in Section 3.
In a final Appendix we show that uniformly lipschitz domains,
bounded and connected, satisfy conditions \ref{h1}-\ref{h3}. This
implies that Theorem \ref{main} generalizes the result of
\cite{chakib}.

\section{Main result}

In this section we first give some preliminary results, which will
be of use later. After that, we shall prove Theorem \ref{main}.

Some notation is in order; given $A \subset \RN$ and $r>0$, we
define:

\begin{equation} \label{dist} \begin{array}{r}  A^r=\{x \in \Omega:\ d(x, \partial A)
>r
\}, \\ A_r=\{x \in \RN:\ d(x,  A) <r\}. \end{array} \end{equation}

We also need a notion of uniformity for the Lipschitz regularity
of a domain. The following definition follows closely that of
Chenais \cite{chenais}.

\medskip We first introduce some notation. Given $z\in \RN$, we
denote $z=(\hat{z}, z_N)$ where $\hat{z}=(x_1,\dots x_{N-1}) \in
\R^{N-1}$ and $z_N \in \R$ give the coordinates of $z$ in a given
coordinate system.

\begin{definition} \label{def:lip} We say that $\Omega$ is a uniform Lipschitz domain with constants
$M>0$, $\gamma>0$ (and we denote $\Omega \in Lip(M,\gamma)$ if the
following is satisfied:

For all $x \in \partial \Omega$, there exists a local coordinate
system satisfying what follows. There exists a function
$\varphi_x: U_{\hat{x}}  \to \R$ a Lipschitz function with
Lipschitz constant smaller than $M$ such that:

$$ y \in O_x \cap \Omega \Leftrightarrow y \in O_x \mbox{ and } y_N > \varphi_x(\hat{y}). $$

Here,

$$U_{\hat{x}}= \{ \hat{z} \in \R^{N-1}:\ |z_i - x_i| < \gamma\},$$
$$O_{x}= \{ z \in \R^{N}:\ |z_i - x_i| < \gamma\ \ i=1 \dots N-1,\ |z^N-x^N| < M \gamma (n-1)^{1/2}\}.  $$

\end{definition}

Observe that for uniformly bounded domains, this definition
coincides with the definition of strong local Lipschitz property
of Adams (\cite{adams}, page 66). This property is equivalent to
the condition assumed in \cite{chakib}, see Section 3 of
\cite{chenais}.

\medskip Now let us recall and comment the following result, due to
Gagliardo \cite{gagliardo} (alternatively, see Adams \cite{adams},
Theorem 4.8).

\begin{theorem} \label{teo:ga} Let $\Omega$ a bounded
domain satisfying the cone condition \ref{h2}. Then there exists a
finite collection $\Omega_1, \dots \, \Omega_k$ of open subsets of
$\Omega$ such that:

\begin{enumerate}
\item $\Omega = \cup_{i=1}^k \Omega_i$.

\item $\Omega_i$ is a Lipschitz domain.

\end{enumerate}

\end{theorem}

\begin{remark} \label{rem:ga} By having a look at the proof of the above
result (\cite{adams}, Theorem 4.8), one obtains easily that the
Lipschitz constants of the graphs depend only on the cone $\C$.
Moreover, the number of sets $A_j$ in the proof can be bounded in
terms of $\C$ and $R$ of condition \ref{h1}. This implies the
following uniform version: under conditions \ref{h1} and \ref{h2},
then $\Omega_i \in Lip(M,\gamma)$, and $k$, $M$, $\gamma$ are
positive constants depending only on the cone $\C$ and the
constant $R$ given by conditions \ref{h1} and \ref{h2}.

\end{remark}

In the following proposition we prove property \ref{q}, which will
be essential in our proof of Theorem \ref{main}.

\begin{proposition} \label{prop:Q} Let us fix $R>0$ and a finite cone
$\C$. Then, property \ref{q} holds for any $\Omega$ satisfying
\ref{h1} and \ref{h2} referred to $R$ and $\C$ respectively.
\end{proposition}

\begin{proof} We will prove the proposition in two steps:

\bigskip {\bf Step 1:} The case $\Omega \in Lip(M, \gamma)$.

This is straightforward. Indeed, consider a graph like in
Definition \ref{def:lip}, and take $t>0$ sufficiently small.
Define

$$W=\{ y \in O_x: y_N > \varphi_x(\hat{y}) + t \}$$

It is clear that for any point $(\hat{y}, \varphi_x(\hat{y}) +
t)$, the distance to the boundary of $\Omega$ is at most $t/M$,
where $M$ is the Lipschitz constant of $\varphi_x$. Therefore,

$$ |O_x \setminus \Omega^{t/M}| \leq |O_x \cap \Omega \setminus W| \leq \gamma^{N-1} t.$$

Observe also that the number of graphs needed to cover $\Omega$ is
uniformly bounded. So it suffices to take $t$ sufficiently small,
according to $\e$. Indeed, it follows that the dependence of
$\delta$ on $\e$ is linear.

\bigskip {\bf Step 2:} The general case, with $\Omega $ satisfying
\ref{h1}, \ref{h2}.

By \ref{teo:ga}, $\Omega \subset \cup_{i=1}^k \Omega_i$ where each
$\Omega_i \in Lip (M, \gamma)$. Obviously, $\partial \Omega \subset
\cup_{i=1}^k
\partial \Omega_i$. Then, it is easy to show that:

$$ \Omega \setminus \Omega^{\delta} \subset \cup_{i=1}^k  \Big( \Omega_i \setminus \Omega_i^{\delta} \Big). $$

Then, we can conclude by Step 1. Since $M,\ \gamma,\ k$ depend
only on $\C$ and $R$, we obtain the uniform estimate.

\end{proof}

Another important tool in our proof is the classical Sobolev
embedding. The uniformity of the Sobolev constant is important for
our purposes and not so well-known, so let us state it in detail.
For a proof, see \cite{adams}, Lemmas 5.10 and 5.15, and also
Theorem 8.25 (observe the remark at the end of the proof of this
last theorem).

\begin{proposition} \label{sobolev} Let $\Omega \subset \RN$ satisfy \ref{h1},
\ref{h2}, and $p>1$. Then there exists $C_S>0$ depending only on
$R,\ p,\ \C$ and $N$ such that:

\begin{enumerate}

\item If $N>p$, $\| u \|_{L^{q}(\Omega)} \leq C_S \|u
\|_{W^{1,p}(\Omega)}$, for any $u \in W^{1,p}(\Omega)$, $q\in [p,
\frac{Np}{N-p}]$.

\item If $N=p$, $\| u \|_{L^{q}(\Omega)} \leq C_S \|u
\|_{W^{1,p}(\Omega)}$, for any $u \in W^{1,p}(\Omega)$, $q\in [2,
+\infty)$.

\noindent Even more, there holds: $\| u \|_{L_A(\Omega)} \leq C_S
\|u \|_{W^{1,p}(\Omega)}$, where $\| \cdot \|_{L_A(\Omega)}$ is
the Orlicz norm associated to $A(t)= e^{t^{\frac{N}{N-1}}}-1$.

\item If $N<p$, then any function $u \in  W^{1,p}(\Omega)$ is
equal almost everywhere to a continuous function. Moreover, $\| u
\|_{L^{\infty}(\Omega)} \leq C_S \|u \|_{W^{1,p}(\Omega)}$.

\end{enumerate}

\end{proposition}

So far, we have considered conditions \ref{h1} and \ref{h2}.
Condition \ref{h3} is less standard and can be considered as a
uniform connectedness condition. In the following example we show
that this condition is indeed necessary for the thesis of the
Theorem \ref{main}.

\begin{example} \label{ejemplo}

In this example we consider $N\geq 3$, and we write the
coordinates of a point in $\RN$ as $x=(\hat{x}, x_N)$, with
$\hat{x} \in \R^{N-1}$ and $x_N \in \R$. Let us define the cone:

$$\C=\{ x \in \RN;\ 0<x_N<1,\ |\hat{x}| < x_N \}.$$

For any $\e>0$, we define the set:

$$ \Omega_\e = \Big(-\e e_N + \C \Big) \cup \Big(\e e_N - \C
\Big),$$ where $e_N=(0, \dots 0,\ 1)$.

Clearly, $\Omega_\e$ is a connected set and the family of sets
$\Omega_\e$ satisfy \ref{h1} and \ref{h2}, but they do not satisfy
\ref{h3}. We shall find a sequence $u_\e \in H^1(\Omega_\e)$ such
that:
$$\int_{\Omega_\e} u_\e \, dx=0, \ \int_{\Omega_\e} u_\e^2 \, dx
\geq c >0,\ \int_{\Omega_\e} |\nabla u_\e|^2 \, dx \to 0.$$

Define $$u_\e(x)=\left \{ \begin{array}{ll} -1 & x_N \leq -\e, \\
\e^{-1} x_N &  |x_N| <\e,\\ 1 & x_N \geq \e. \end{array} \right.$$

Since $u_\e$ is odd with respect to the $x_N$ variable,
$\int_{\Omega_\e} u_\e \, dx=0$. It is also easy to show that:
$$ \int_{\Omega_\e} u_\e^2 \, dx \to 2 |\C| \mbox{ as } \e \to
0.$$

Finally, we compute:

$$ \int_{\Omega_\e} |\nabla u_\e|^2 \, dx = \int_{\Omega_\e \cap \{-\e<x_N<\e\}} \e^{-2} \,
dx \leq 2 \e^{-2} |\C \cap \{0<x_N<2\e\}| = C \e^{N-2}.$$

\end{example}

\bigskip

Now we are in conditions to prove our main result:

\bigskip

\begin{proof}[Proof of Theorem \ref{main}]

Choose $\e>0$ arbitrarily small, to be fixed at the end of the
proof. Take $\delta \in (0, \d_0)$ such that $|\Omega \setminus
\Omega^{\delta}|<\e$, as given by property \ref{q}. Fix also $r
\in (0, \delta/2)$.

Consider a covering $B(0,R)=\cup_{i=1}^M B_i$, where $B_i$ are
open balls of radius $r$. We denote by $\mathcal{A}$ the class of
connected sets formed by unions of such balls; in other words:

$$ \mathcal{A} = \{ A= \cup_{i \in J} B_i,\ J \subset \{1,\dots M\},\
A \mbox{ is connected}\}.$$

Clearly each element of $\mathcal{A}$ satisfies a cone condition
and hence each of them has a Poincar{\'e} constant $C_A>0$, in virtue
of Theorem \ref{poincare}. In other words,

$$ C_A \int_{A} |\nabla u|^p\, dx \geq \int_A |u|^p\, dx \ \ \forall u \in W^{1,p}(A),\ \int_{A} u\, dx=0.  $$

Since $\mathcal{A}$ is finite, take $\tilde{C}$ the maximum of all
those constants. Observe that $\tilde{C}$ depends only on $r$, and
hence on $\e$.

Define:

$$ J=\{i \in \{1, \dots M\}:\ B_i \cap \Omega^{\delta} \neq \emptyset \}, \ \
A=\cup_{i\in J} B_i.$$

Observe that under our assumptions $\Omega \supset A \supset
\Omega^{\delta}$. Moreover, it is clear that $A$ is connected, and
then $A \in \mathcal{A}$.

Take now $u\in W^{1,p}(\Omega)$ with $\int_{\Omega} u \, dx=0$,
$\int_{\Omega} |u|^p\, dx =1$. Our intention is to find a lower
bound for $\int_{\Omega} |\nabla u|^p\, dx$. Let us assume that
$\int_{\Omega} |\nabla u|^p\, dx \leq 1$, otherwise we are done.

We first give an estimate on $\int_A u\, dx$ by using H{\"o}lder
inequality:

\begin{equation} \label{estim0} \left | \int_A u\, dx \right | = \left | \int_{\Omega\setminus A} u\, dx \right
|\leq |\Omega\setminus A|^{\frac{p-1}{p}} \left (\int_{\Omega} |u|^p
\, dx \right )^{1/p} \leq \e^{\frac{p-1}{p}} .\end{equation}

Let us denote by $u_A$ the average of u on $A$. Then,

\begin{equation} \label{uA} |u_A| = |A|^{-1} \left |\int_A u\, dx \right | \leq
\frac{\e^{\frac{p-1}{p}}}{|A|}.\end{equation}

Observe also that $|A| \geq | \Omega^\delta| = |\Omega| - |\Omega
\setminus \Omega^\delta| \geq V- \e$, where $V$ is the volume the
cone $\C$, which serves as a lower bound for the volume of
$\Omega$. We choose $\e$ small enough so that $|u_A|<1$.

Our intention now is to find a lower estimate for the term $\int_A
|u-u_A|^p\, dx$. Observe that the mean value theorem implies that:

\begin{equation} \label{mvt} \Big ||u-u_A|^p - |u|^p \Big| \leq C_p (|u|^{p-1}+1) u_A. \end{equation}

We now use again H{\"o}lder inequality to estimate the term:

\begin{equation} \label{estim1} \int_A
|u|^{p-1}\, dx \leq \left (\int_A |u|^{p}\, dx
\right)^{\frac{p-1}{p}} |A|^{\frac 1 p} \leq |\Omega|^{\frac 1 p}.
\end{equation}

Therefore,

\begin{equation} \label{estim2} \Big | \int_A  (|u-u_A|^p - |u|^p)\, dx \Big| \leq C_p
\frac{\e^{\frac{p-1}{p}}}{V-\e} (|\Omega|^{\frac 1 p}+ |\Omega|).
\end{equation}

It suffices now to estimate $\int_A  |u|^p\, dx$. Let us take any
$q>p$ so that the Sobolev embedding holds, see Proposition
\ref{sobolev}. Since we are assuming that $\int_{\Omega} |\nabla
u|^p\, dx \leq 1$, we have that $\|u\|_{L^q(\Omega)} \leq C_S
2^{1/p}$. By using again H{\"o}lder inequality, we obtain:

\begin{equation} \label{estim3} \int_A |u|^p\, dx = 1 - \int_{\Omega \setminus A} |u|^p\, dx
\geq 1 - \| u \|_{L^q(\Omega)}^p \left | \Omega \setminus A
\right|^{\frac{q-p}{q}} \geq 1- 2 C_S^p \, \e^{\frac{q-p}{q}}.
\end{equation}

By putting together \eqref{estim2} and \eqref{estim3}, we have
that:

$$ \int_A |u-u_A|^p\, dx \geq 1- 2 C_S^p \, \e^{\frac{q-p}{q}} - C_p
\frac{\e^{\frac{p-1}{p}}}{V-\e} (|\Omega|^{\frac 1 p}+
|\Omega|).$$

We only need to chose $\e$ small enough so that the above term is
greater than $1/2$. Therefore,

$$ \int_{\Omega} |\nabla u|^p\, dx \geq \int_{A} |\nabla (u-u_A)|^p\,
dx\geq \frac{1}{2\tilde{C}}.$$

\end{proof}

As a corollary, we obtain the following result.
\begin{corollary} Given $E \subset \Omega$ a measurable set with nonzero measure,
we denote by $u_E$ the average of u in $E$, that is, $u_E =
|E|^{-1} \int_E u\, dx$.

There exists a constant $C>0$ such that for any domain $\Omega$
satisfying \ref{h1}-\ref{h3},
\begin{enumerate}
\item If $N>p$, $$\displaystyle  \int_{\Omega} | u - u_E|^p\, dx
\leq C |E|^{\frac{p-N}{N}} \int_{\Omega} |\nabla u|^p\, dx \ \
\forall u \in W^{1,p}(\Omega), \ E \subset \Omega, \ |E|>0.$$

\item If $N=p$, $$\displaystyle  \int_{\Omega} |u - u_E|^p\, dx
\leq C (\log(1+ |E|^{-1}))^{N-1}  \int_{\Omega} |\nabla u|^p\, dx
\ \ \forall u \in W^{1,p}(\Omega), \ E \subset \Omega, \ |E|>0.$$

\item If $N<p$, $$\displaystyle  \int_{\Omega} |u - u(x_0)|^p\, dx
\leq C  \int_{\Omega} |\nabla u|^p\, dx \ \ \forall u \in
W^{1,p}(\Omega), \ x_0 \in \Omega.$$

\end{enumerate}

Moreover, $C$ depends only on $p$, $R, \e$ and the cone
$\C$ given in conditions \ref{h1}-\ref{h3}.

\end{corollary}

\begin{remark} Compared
with Theorem 1 and Corollary 2 of \cite{chakib}, here the constant
$C$ is independent of the set $E$. The dependence of the estimate
on $E$, when $|E|$ is small, is made explicit here.

\end{remark}

\begin{proof} Obviously, it suffices to prove the estimate for $u \in W^{1,p}(\Omega)$
with $\int_{\Omega} |\nabla u|^p\, dx=1$ and $\int_{\Omega}u\,
dx=0$. By Theorem \ref{main}, there exists $C>0$ (depending only
on $R$, $\C$ and $\d_0$) such that $\int_{\Omega}|u|^p\, dx \leq
C$.

Observe that:

\begin{equation} \label{hola} \int_{\Omega} |u-\lambda|^p\, dx \leq 2^p \left ( \int_{\Omega} |u|^p \, dx + \lambda^p |\Omega|\right ).\end{equation}

Therefore it suffices to estimate $u_E$ in cases 1 and 2, and
$u(x_0)$ in case 3. For that we will use the Sobolev embedding;
observe that $\|u\|_{W^{1,p}} \leq (1+C)^{1/p}$.

\bigskip {\bf Case 1:}  If $N>p$, we have:

$$ \left | \int_E u \, dx \right | \leq \left ( \int_E
|u|^{\frac{Np}{N-p}}\, dx \right )^{\frac{N-p}{Np}} \left ( \int_E
1\, dx \right )^{1+\frac{1}{N} - \frac{1}{p}}  \leq C'
|E|^{1+\frac{1}{N} - \frac{1}{p}},$$

where $C'$ depends on $C$ and $C_S$. But then:

\begin{equation} \label{hola1} |u_E| \leq C' |E|^{\frac{1}{N} - \frac{1}{p}}. \end{equation}

\bigskip {\bf Case 2:} If $N=p$, we need to use Orlicz norms: we refer to
\cite{adams}, Chapter 8. Let us define $A(t)=
e^{t^{\frac{N}{N-1}}}-1$, and $\tilde{A}(t)$ its complement
function (which is not explicit). By using the generalized H{\"o}lder
inequality,

$$\left |\int_E u\, dx \right |= \left |\int_{\Omega} u \chi_E \, dx \right |\leq 2 \| u \|_{L_A}
\| \chi_E \|_{L_{\tilde{A}}} \leq C' \| \chi_E \|_{L_{\tilde{A}}}
.$$

Here $\chi_E$ is the characteristic function of the set $E$. We
now use the definition of the Orlicz norm:

$$ \| \chi_E \|_{L_{\tilde{A}}} = \inf_{k>0} \left \{
\int_{\Omega} \tilde{A} \left( \frac{\chi_E(x)}{k} \right) \, dx
\leq 1 \right \} =  \inf_{k>0} \left \{  \tilde{A} \left(
\frac{1}{k} \right) |E| \leq 1 \right \} $$$$ =
\frac{1}{\tilde{A}^{-1}(|E|^{-1})} \leq
\frac{A^{-1}(|E|^{-1})}{|E|^{-1}}= |E| \left (\log\left
(1+|E|^{-1} \right ) \right)^{\frac{N-1}{N}}.$$

Above we have just used the general inequality $s \leq A^{-1}(s)
\tilde{A}^{-1}(s).$

Therefore,

\begin{equation} \label{hola2} |u_E| \leq C' \left (\log\left
(1+|E|^{-1} \right ) \right)^{\frac{N-1}{N}}. \end{equation}

\bigskip {\bf Case 3:} If $N<p$, then $u$ is continuous. Moreover,

\begin{equation} \label{hola3} u(x_0) \leq \|u\|_{L^{\infty}(\Omega)} \leq C',\end{equation} by the uniformity of the constant on
the Sobolev inequality.

\bigskip Putting together estimates \eqref{hola}, \eqref{hola1},
\eqref{hola2}, \eqref{hola3}, we finish the proof.

\end{proof}
\begin{remark}
One can extend also the previous results to
sets $E$ which are hypersufaces, in the spirit of Theorem 2 of \cite{chakib}. We leave the details to the interested reader.
\end{remark}

\section{Some generalizations}
In this section we discuss some possible extensions of our results
to more general frameworks. First, we will discuss the possible
relaxation of the hypotheses of Theorem \ref{main}. In particular,
we will prove Corollary \ref{coro}. Finally, we will briefly
comment the extension of our results to Riemannian manifolds.

\subsection{Relaxation of conditions \ref{h1}-\ref{h3}}

The cone condition \ref{h1} has been needed in our proof at two
different steps. First, it is needed to guarantee a Sobolev
inequality (see Proposition \ref{sobolev}), together with the
uniformity of the constant. Observe that the proof of Theorem
\ref{main} only needs a Sobolev inequality for some exponent
$q>p$; the limiting exponent is not needed here.

Second, \ref{h1} is used to derive property \ref{q} given in
Proposition \ref{prop:Q}, which otherwise should be somehow
imposed. Indeed, we can introduce it in a generalized version of
condition \ref{h3}.

Our result can be generalized as follows.

\begin{theorem} \label{general} Take $p>1$, and let $\F$ be a family of domains satisfying the following
hypotheses:

\begin{enumerate}[label=(f\arabic*), ref=(f\arabic*)]
\setcounter{enumi}{0}

\item \label{f1} There exists $R>0$ such that $\Omega \subset
B(0,R)$ for any $\Omega \in \F$.

\item \label{f2} There exists $q>p$ and $C_S>0$ such that for any
$\Omega \in \F$, $u \in W^{1,p}(\Omega)$,

$$\| u \|_{L^{q}(\Omega)} \leq C_S \|u \|_{W^{1,p}(\Omega)}.$$

\item \label{f3} For any $\e>0$ there exists $\delta=\delta(\e)
>0$ such that for any $\Omega \in \F$ we can choose a connected set $U \subset \Omega$ such that
$U \subset \Omega^{\delta}$ and $|\Omega \setminus U| <\e$.

\end{enumerate}
%
%
%

Then, there exists $C>0$ such that for any $\Omega \in \F$,

$$ \int_{\Omega} |u|^p\, dx \leq C \int_{\Omega} |\nabla u|^p\, dx \
\forall u \in W^{1,p}(\Omega),\ \int_{\Omega} u(x) \, dx=0.$$

\end{theorem}

The proof of Theorem \ref{general} is basically the same as the
proof of Theorem \ref{main}, with the obvious modifications due to
the new hypotheses. The details are left to the reader.

The literature on Sobolev inequalities for domains not satisfying
the cone property is huge. For instance, it is known that the
limiting Sobolev inequality does not hold for domains containing
outward cusps. However, for cusps having power sharpness, a
Sobolev embedding to some $L^q(\Omega)$ space is possible, see
\cite{adamspaper}. It is to be expected that some uniformity on
the assumptions would imply an uniform constant in that
inequality, providing \ref{f2}.

Many other Sobolev embeddings have also been found, and it is not
possible to give here a detailed account of them: see for instance
\cite{besov, mazya} and the references therein. 

On the other hand, condition \ref{f3} puts together condition
\ref{q} and \ref{h3} in a more general fashion. Condition \ref{f3}
does not seem easy to check in general, though. However, as a
particular case, we can prove Corollary \ref{coro}.

It is not clear to us whether a domain under the conditions of
Corollary \ref{coro} satisfies condition \ref{h3}. In any case,
Corollary \ref{coro} is a consequence of Theorem \ref{general}, as
we show below.

\begin{proof}[Proof of Corollary \ref{coro}]
We only need to check that under the conditions of Corollary
\ref{coro}, we can apply Theorem \ref{general}. Let us define:

$$ \F=\{ K_r: K \subset B(0,R) \mbox{ a connected set} \}. $$

We recall \eqref{dist} for the definition of $K_r$. In what
follows we can always consider $K$ to be a compact set; otherwise,
we would argue on its closure. Clearly, any domain in $\F$
satisfies the cone property with a cone depending only on $r$.
Therefore, conditions \ref{f1} and \ref{f2} are satisfied.

Let us check condition \ref{f3}. Given $\delta \in (0,r/2)$ and $K
\subset B(0,R)$ connected, we define:

$$U= \overline{K_{r-\delta}}= \{ x \in \RN:\ d(x, K) \leq r-\delta \}.$$ Clearly, $U$ is a connected set. Observe now
that $\partial K_r = \{x \in \RN:\ d(x,K)=r \}$. The triangular
inequality readily implies that $U \subset K_r^{\delta}$. In order
to prove \ref{f3}, it suffices to estimate the measure of $K_r
\setminus U$ independently of the choice of $K$.

Let us point out that:

$$ K_r \setminus U = \{x \in \RN:\ d(x,K) \in (r-\delta, r) \}= \{x \in \RN:\ d(x,K_{r-\delta}) \in (0, \delta) \}. $$

The first equality holds just by the definition of $U$; let us
justify briefly the second identity. The inclusion $\supset$
follows easily from the triangular inequality. For the inclusion
$\subset$, take $x\in \RN$ and $y \in K$  with $d(x,y)=d(x,K) \in
(r-\delta,r)$. Choose $z$ in the segment $[x,y]$ such that $d(x,z)
\in (0,\delta)$, $d(z,y) < r-\delta$. Then, $z \in K_{r-\delta}$
and also $d(x,K_{r-\delta}) \leq d(x,z) < \delta$. It suffices to
observe that $x \notin U$ to conclude.

Now, observe that $K_{r-\delta}$ satisfies the cone condition with
a fixed cone. Now it suffices to recall the arguments of the proof
of Proposition \ref{prop:Q} to conclude.

\end{proof}

\subsection{Extension to Riemannian manifolds}
Because of our original motivation in the study of local versions
of the Moser-Trudinger inequality, we are interested in a uniform
Poincar{\'e} inequality in compact Riemannian manifolds. This
extension is quite direct, but it deserves a couple of comments.

First of all, we have not been able to find a specific definition
of the interior cone condition for a domain in a manifold. Still,
we think that it must be written elsewhere, but since we cannot
cite a reference we give below a definition which is useful for
our purposes.

\begin{definition} Let $\Sigma$ be a compact manifold. It is clear that we can
consider on $\Sigma$ a finite family of charts $x_i:B(0,1) \to
\Sigma$, $i=1, \dots k,$ such that $\cup_{i=1}^k x_i(B(0,1/2)) =
\Sigma$. We say that $\Omega \subset \Sigma$ satisfies the
interior cone condition if for any $j$ such that
$x_j^{-1}(\partial \Omega) \cap B(0,1/2) \neq \emptyset$,
$x_j^{-1}(\Omega)$ satisfies the interior cone condition at any
point of $x_j^{-1}(\partial \Omega) \cap B(0,1/2)$.

\end{definition}

By using this finite family of charts it is easy to generalize
Proposition \ref{q} to the case of sub-domains of a fixed compact
Riemannian manifold (instead of \ref{h1}). Moreover, also the
uniform Sobolev embedding holds, as can be easily checked passing
to charts. Therefore, the results of Theorem \ref{main}, Theorem
\ref{general} and Corollary \ref{coro} follow with basically the
same proof.

\section{Appendix}

In Theorem \ref{main} we have proved the uniformity of the
Poincar{\'e} constant for domains satisfying certain properties,
namely \ref{h1}, \ref{h2}, \ref{h3}. In this final section we show
that connected uniformly Lipschitz domains (in the sense of
Definition \ref{def:lip}) satisfy \ref{h2} and \ref{h3}. So, our
result extends that of \cite{chakib} to a less regular framework.

Property \ref{h2} being straightforward, it suffices to show
property \ref{h3}. The proof we give here is inspired by min-max
theory and deformation flows. It is possible that there is a more
direct way to prove this result, but in our opinion the proof
itself is interesting in its own right.

Take $\Omega$ be a uniformly Lipschitz domain, according to
definition \ref{def:lip}. Let us define $\dd: \R^N \to \R$ as:

\begin{equation} \label{elena} \dd(x)= \left \{\begin{array}{lr} d(x, \partial \Omega) &
\mbox{ if } x \in \overline{\Omega},\\ -d(x, \partial \Omega) &
\mbox{ if } x \notin \overline{\Omega}, \end{array} \right.
\end{equation}

Observe that with that definition, $\Omega = \dd^{-1}(0,+\infty)$
and $\Omega^{\delta}= \dd^{-1}(\delta, +\infty)$  (for
$\delta>0$). The idea of the proof is to show that those two sets
are homotopically equivalent if $\delta \in (0,\delta_0)$, for
some $\delta_0$ small independent of $\Omega$. In order to prove
that, we will use a deformation argument.

It is important to observe here that $\dd$ is a Lipschitz map, but
not $C^1$, and therefore one cannot use a gradient flow as
usually. However, in 1981, K. C. Chang (\cite{kcchang}) was able
to apply variational methods to Lipschitz maps. For that, he
needed a notion of generalized gradient, due to Clarke, that we
remind below. Moreover, he was able to prove the existence of a
continuous pseudogradient flow for Lipschitz maps. With these two
main tools, he was able to apply min-max techniques to Lipschitz
functionals.

Here we just use those ideas to find an homotopical equivalence
between $\Omega$ and $\Omega^\delta$. For that, we will need to
show that $\dd$ has no critical points in $\Omega \setminus
\Omega^\delta$; it is here that the Lipschitz assumption on
$\Omega$ comes into play.

First, let us recall the notion of generalized gradient of a
Lipschitz map due to Clarke. For more information on those
aspects, see \cite{clarke, clarke2, kcchang}.

\begin{definition} Let $f:\R^N \to \R$ a Lipschitz map. The
generalized directional derivative of $f$ on $x \in \R^N$ along
the direction $v \in \R^N$ is defined as:

$$ f^0(x;v)= \limsup_{y \to x, t \to 0^+} \frac{f(y+tv)-f(y)}{t}.$$

It can be checked that the function $v \mapsto f^0(x;v)$ is
continuous and convex, and that $|f^0(x;v)| \leq L |v|$, where $L$
is the Lipschitz constant for $f$. We define the generalized
gradient of $f$ at $x$, denoted $\partial f(x)$, as the
subdifferential of the map $v \mapsto f^0(x;v)$ at $0$. In other
words,
$$w \in \partial f(x) \Leftrightarrow w\cdot v \leq f^0(x;v) \
\forall v \in \R^N.$$

For any $x$, $\partial f(x)$ is a non-empty compact and convex
subset of $\RN$.

Finally, we say that $x$ is a critical point of $f$ if $0 \in
\partial f(x)$.

\end{definition}

\begin{lemma} Let $\Omega$ be a uniformly lipschitz domain,
with the definition \ref{def:lip}, and $\dd$ the map defined in
\eqref{elena}. Then, there exists $\delta_0>0$ depending only on
the constants $M$, $\gamma$ such that $\dd$ has no critical points
in $\overline{\Omega} \setminus \Omega^{\delta}$.
\end{lemma}

\begin{proof}

Take $x \in \Omega \setminus \Omega^{\delta}$; by the definition
of Lipschitz regularity, in a neighborhood of $x$, $\partial
\Omega$ is a graph on the last component (with a conveniently
chosen coordinate system). Take $v= (0, \dots 0, 1)$. We will show
that there exists $c>0$ such that $\dd^0(x;v)<-c<0$; this readily
implies that $\dd$ has no critical points in $\overline{\Omega}
\setminus \Omega^{\delta}$.

By making $z= y + t v$, we have that:

$$ \dd^0(x;v)= \limsup_{z \to x, t \to 0^+} \frac{\dd(z)-\dd(z-tv)}{t}.$$

Take $r=\dd(z) \in (0,\delta]$. Clearly, $B(z,r) \subset \Omega$.
Let us define the cone:

$$\C=\{ x \in \RN;\ 0>x_N>-\tau,\ |\hat{x}| < M x_N \}.$$

Here $\tau>0$ is chosen depending on $M$, $\gamma$. By the
Lipschitz property of the domain, $y+\C \subset \Omega$ for any $y
\in B(z,r)$. In other words, $(z+\C)_r \subset \Omega$ (see
\eqref{dist}). An easy geometrical argument implies the following
(see Figure 1):

$$d(z-tv, \partial \Omega) \geq d(z-tv, \partial (z+\C)_r) \geq r+ct \ \ \mbox{ if } t \mbox{ is small.}$$

Here $c$ can be made explicit, $c=\frac{1}{\sqrt{1+M^2}}>0$. This
finishes the proof.

\begin{figure}[h]
\centering \fbox{
\begin{minipage}[c]{80mm}
           \centering
        \resizebox{80mm}{70mm}{\includegraphics{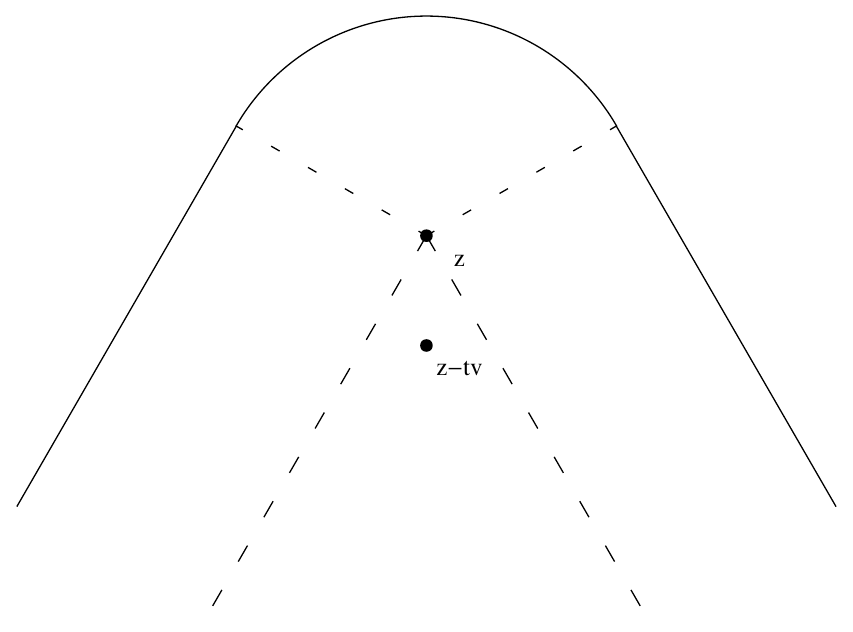}}
\caption{}

         \end{minipage}
       }
\end{figure}

\end{proof}

\begin{proposition} Let $\Omega \subset B(0,R)$ be a uniformly lipschitz domain,
with the definition \ref{def:lip}, and $\dd$ the map defined in
\eqref{elena}. Then, there exists $\delta_0>0$ depending only on
the constants $R$, $M$, $\gamma$ such that $\Omega^\delta$ is
homotopically equivalent to $\Omega$ for all $\delta \in
(0,\delta_0)$. In particular, if $\Omega$ is connected,
$\Omega^{\delta}$ is connected.

\end{proposition}

\begin{proof}

The proof is a direct application of the arguments in
\cite{kcchang}. In that paper, Lemma 3.3 gives a Lipschitz
pseudogradient vector field and Lemma 3.4 uses it to define a
convenient deformation, which provides us with the desired
homotopy. Observe that the (PS) condition holds since we are in a
compact framework without critical points.

\end{proof}

\end{document}